\documentclass[12pt]{article}
\usepackage[utf8]{inputenc}
\usepackage{amsmath}
\usepackage[T2A]{fontenc}
\usepackage[english]{babel}
\usepackage{hyphenat}
\usepackage{floatrow}

\usepackage{graphicx}
\graphicspath{ {./Images/} }
\usepackage{array}
\usepackage{caption}
\usepackage{amsfonts, amsthm, xcolor}
\usepackage{url}
\usepackage[margin=1in, total={8.5in, 11in}]{geometry}
\usepackage{tikz}
\date{March 2023}

\newenvironment{claim}[1]{\par\noindent\underline{Claim:}\space#1}{}

\newcommand{\R}{\mathbb{R}}

\newcommand{\Z}{\mathbb{Z}}

\title{Tilings of $\Z$ with multisets of distances}
\author{Andrey Kupavskii\footnote{Moscow Institute of Physics and Technology, Russia; Email: {\tt kupavskii@ya.ru}}, Elizaveta Popova\footnote{Moscow Institute of Physics and Technology, Russia; Email: {\tt popova.es@phystech.edu}}}
\newtheorem{Theorem}{Theorem}
\newtheorem{Lemma}[Theorem]{Lemma}
\newtheorem{Statement}[Theorem]{Proposition}
\newtheorem{Corollary}[Theorem]{Corollary}

\begin{document}
\maketitle

\begin{abstract}
    In this paper, we study  tilings of $\Z$, that is,   coverings of $\Z$ by disjoint sets (tiles). Let $T=\{d_1,\ldots, d_s\}$ be a given multiset of distances. Is it always possible to tile $\Z$ by tiles, for which the multiset of distances between consecutive points is equal to $T$? In this paper, we give a sufficient condition that such a tiling exists. Our result allows multisets of distances to have arbitrarily many distinct values. Our result generalizes most of the previously known results, all of which dealt with the cases of $2$ or  $3$ distinct distances.
\end{abstract}
\section{Introduction}

Tilings is a classical research topic in combinatorics. In general, tiling questions are as follows: is it possible to partition a given set $S$ into disjoint subsets from some special collection? This collection often consists of all images of one given set under the action of some group. For instance, $S$ may be equal to $\Z^n$, and the goal may be  to tile it by translations of a given set $T$, usually called  a {\it tile}. Most studied in this direction is the one-dimensional case \cite{Bhat}, \cite{CM}, \cite{New2}, where the main task (solved only in some special cases) is to classify all sets that tile $\Z$.
A modification of this problem with group of translations replaced by all isometries of $\mathbb{Z}$ was also considered. In other words, the question is if there exists a tiling of $\mathbb{Z}$ by translations of a given tile $T$ and its mirror reflection. It is known that translations of any $3$-point tile a finite interval in $\mathbb{Z}$  (in particular, see \cite{Hons}, where the author suggests a very elegant algorithm for finding such a tiling). On the other hand, it is easy to construct a $4$-point tile, for which it is false (see \cite{Naka}).

In his paper \cite{Naka},  suggested the following interesting variant of the problem. Let us call the multiset $\{v_{i}-v_{i-1}\}_{i=1}^m$ a gap set of a tile $(v_0,...,v_m)$, where points are listed in increasing order. If a gap set contains values $d_i$ with multiplicities $k_i$, $i=1, ..., s$, then we denote it  $\{d_1^{(k_1)}, d_2^{(k_2)}, ..., d_s^{(k_s)}\}$.
Does there exist a tiling of $\mathbb{Z}$ (or of a finite interval in $\Z$, which is a stronger property) by translations of all tiles with a given gap set? That is, the group of allowed transformation of an initial tile contains all permutations of gap lengths (distances between consecutive points), along with the translations of $\mathbb{Z}$. We will further say, if a tiling exists, that ``the gap set tiles $\Z$ (or interval)''.

Note that, for $3$-point tiles, this new question is the same as the previous one concerning isometries. Indeed, the only nontrivial permutation of two gaps of such a tile gives the mirror reflection of that tile. For tiles containing strictly more than three points, this new group of transformations is significantly richer. To the best of our knowledge, no gap set that answers this question in the negative has been found. In the paper \cite{CJK} it is proved that the tiling always exists for tiles of four points with one gap length being sufficiently large in comparison to other:

\begin{Theorem}[Choi, Jung, Kim\cite{CJK}]\label{thmcjk}
    Let $p,q,r$ be positive integers. The gap set $\{p,q,r\}$ tiles an interval in $\Z$ provided $r\ge 63(\max\{p,q\})^2.$
\end{Theorem}

In \cite{Naka} the following results are proved for gap sets with only two distinct values.

\begin{Theorem}[Nakamigawa \cite{Naka}]\label{thmnaka1}
  Let $p,q, k, \ell$ be positive integers. Then the gap set $\{p^{(k)}, q^{(\ell)}\}$ tiles an interval in $\Z$ if one of the following conditions is satisfied:
  \begin{enumerate}
      \item $k=1$;
      \item $k(k+1)p\le q$;
      \item $k\le \ell$ and $(k+1)p\le q.$
  \end{enumerate}
\end{Theorem}
Cases 2 and 3 of Theorem~\ref{thmnaka1} are corollaries of the following more general result.
\begin{Theorem}[Nakamigawa \cite{Naka}]\label{thmnaka2}
 Let $k, \ell$ be positive integers. Let $a$ be such that both $a$ and $a+1$ can be represented as $\sum_{i=1}^{\ell+1} c_i (k+i)$ with $c_i$ being non-negative integers. Then, provided $p$ and $q$ satisfy $ap \leq q \leq (a+1)p$, the gap set $\{p^{(k)},q^{(\ell)}\}$ tiles an interval in $\Z$.

\end{Theorem}

In this paper, we prove the result that generalizes previous results to the case of more distinct elements in the gap set. 
Roughly speaking, we show that if the gap set grows sufficiently fast and the multiplicities of the smallest and the largest gap lengths are ''relatively big'', the tiling does exist.

\begin{Theorem}\label{thm1}

Let $s \geq 2$, $p,$ $\{d_i\}_{i = 1}^{s+p}, \{k_i\}_{i=1}^{s+p}$ be positive integers, $\sum_{i=3}^{s} k_i + 1\leq k_1$, $\sum_{i = s+1}^{s+p-1} k_i + 1\leq k_{s+p}$, $d_2 \geq d_1((k_1+k_2)^2+1)$ and for each $i = 3,..., s+p$ we have $d_i \geq g_i(d_1,...,d_{i-1},k_1,..., k_i),$  where $g_i$ are some functions. Then the gap set $\{d_1^{(k_1)}, d_2^{(k_2)},..., d_{s+p}^{(k_{s+p})}\}$ tiles an interval in $\Z.$

\end{Theorem}
In the proof we use and generalize methods, elaborated in \cite{Naka}, along with some new combinatorial ideas.
The formulation of Theorem~\ref{thm1} is quite technical, and so we present a corollary that relates it to previous results. Substituting $s=2$, $p=1$ we have both inequalities for multiplicities trivially satisfied and thus the following holds:

\begin{Corollary}

Let positive integers $d_1,d_2,d_3$ and $k_1,k_2,k_3$ be such that $d_2 \geq d_1((k_1+k_2)^2+1)$ and $d_3\geq g(d_1,d_2,k_1,k_2,k_3)$. 
Then gap set $\{d_1^{(k_1)}, d_2^{(k_2)}, d_3^{(k_3)}\}$ tiles an interval.
\end{Corollary}
One can see that it generalizes Theorems~\ref{thmcjk} and~\ref{thmnaka1} modulo the exact bounds on the growth of  gap lengths, which we did not try to optimize for the sake of clarity.

Finally, we note that several related problems were also considered. For instance, in \cite{BJR, FKS, New, Schtu} optimal packings and coverings of $\Z$ using translates/reflections of a fixed set are studied. In \cite{Lead} the following statement is proved: for every tile $T$ in $\Z^m$ there exists $d$, such that isometric copies of $T$ tile $\Z^d$.

\section{Proof of the Theorem~\ref{thm1}}

Note that throughout this section all the conditions of Theorem~\ref{thm1} are assumed.

In the proof, we need to work with both points and vectors. 
We use $v_i$ for points and $x_i$ for vectors. Let $e_1,e_2$ stand for the standard basis vectors $(1,0)$ and $(0,1)$, respectively. For a positive integer $q$, $qe_1$ stands for a vector $ (q,0).$ For positive integers $n,a,b$ we use notation $[n]:= \{1,\ldots, n\}$ and, more generally, $[a,b] = \{a,a+1,\ldots,b\}$.

Let us call a sequence $(v_0,...,v_n)$ in $\mathbb{Z}^2$ a {\it path}, if for each $i\in [n]$ both coordinates of $v_i$ are greater than or equal to the corresponding coordinates of $v_{i-1}$. 
Let us say that $(v_0,...,v_n)$ is a path of type $\{x_1 ^{(k_1)},..., x_\ell ^{(k_\ell)}\}$, if $\{ v_i - v_{i-1}: i\in [n]\} = \{x_1 ^{(k_1)},..., x_\ell ^{(k_\ell)}\}$ as  multisets.  Note that coordinates of these vectors are non-negative by the definition of a path.

Denote $T_i := \{d_1^{(k_1)}, d_2^{(k_2)}, ..., d_{i}^{(k_{i})}\}$ and $V_i := \{(d_1 e_1)^{(k_1)}, ..., (d_{i-1} e_1)^{(k_{i-1})}, e_2^{(k_{i})}\}$ for all $i = 1,...,s+p$. Thus, a path of type $V_i$ has $k_{i}$ vertical gaps and its projection on the horizontal axis is a tile with gap set $T_{i-1}$. 

In the proof, we will several times use the following proposition from \cite{Naka}.

\begin{Statement}[\cite{Naka}]\label{propnaka}
Let $k$ and $\ell$ be positive integers. Then for every positive integer $m$ such that $k+1\leq m \leq \ell+k+1$ there exists  a positive integer $f_{k,\ell}(m)$ such that the rectangle $[0,m-1]\times [0,f_{k,\ell}(m)-1]$ can be tiled by paths of type $\{e_1^{(k)}, e_2^{(\ell)}\}.$
\end{Statement}

We will also need the following fact.

\begin{Statement}\label{lem0}
For a positive integer $k$, if $a \geq k(k-1)$ is an integer, then there exist integers $b\geq 0$, $c\geq 0$ such that $a = b(k+1) + ck$. Consequently, if $a \geq k^2 + 1$, then there exist integers $b\geq 1$, $c\geq 0$ such that $a = b(k+1) + ck$.
\end{Statement}

\subsection{Sketch of the proof}
The proof proceeds by induction on the number of distinct elements in the sequence. We have two different approaches, one for distances up to $d_s$, and the other one for distances $d_{s+1},\ldots, d_{s+p}.$

The base case is  the tiling of some interval by tiles with gap set $T_2$. We obtain it from a tiling of a rectangle of width $d_2$ by paths of type $V_2$ similarly to \cite{Naka}. We also need to make some technical modifications that shall allow us to use  induction. 

The induction step for $\ell = 3,\ldots, s$ is to construct a tiling of some interval by tiles with the gap set $T_\ell$ from the tiling of an interval $[0, L]$ by tiles with the gap set $T_{\ell-1}$. Namely, for each tile $\Gamma$ from the `$\ell-1$-tiling' we make a tiling of the Cartesian product of $\Gamma$ and some interval by paths of type $V_\ell$. This gives us a tiling of a rectangle of width $L+1$ by paths of type $V_\ell$. Next, we introduce a twist that allows us to use Proposition~\ref{lem0}. In particular, we add the point $L+1$ to one of the tiles from step $\ell-1$ in such a way that the new set still admits the tiling of its Cartesian product with some interval by paths of type $V_\ell$. 

As a result, we get a tiling of a rectangle of width $L + 2$ by paths of type $V_\ell$. Since $L+1$ and $L+2$ are relatively prime, we can build a rectangle of any sufficiently large width using the rectangles of widths $L+1$ and $L+2$. Due to technical reasons, we also have to tile a rectangle of width $L + k_\ell d_1 + 1$, similarly to the case of width $L + 2$, by adding points to the tiles from the initial tiling. After that, we make a rectangle of width $d_\ell$ (tiled by paths of type $V_\ell$) using rectangles of these three widths. The rectangles of width $L+k_\ell d_1+1$ are needed in order to make the induction work on the next step: we need to assure that the paths that cover the upper right corner of the rectangle start with $k_1-\sum_{i=3}^\ell k_i$ gaps $d_1e_1$. 
We conclude the induction step by sending the rectangle to the interval by the ''lexicographical'' map $(a,b) \mapsto a + d_\ell b$, that maps a path of type $V_\ell$ to a tile with the gap set $T_\ell$.

The second part of the proof, in which we treat distances $d_{\ell+1},\ldots, d_{\ell+p}$ has similar induction that employs tilings of rectangles of relatively prime width. We, however, use a different technique to obtain such tilings. Namely, 
we use {\it $V_\ell$-homogenous paths} and {\it sequences}: the paths such that every $|V_\ell| + 1$ consecutive points of them form a path of type $V_\ell$, and sequences in $\Z$ such that every $|T_\ell| + 1$ consecutive points of them form a tile with gap set $T_\ell$ (the latter are called {\it $T_\ell$-homogeneous sequences}). On each step we make a tiling of an interval by $T_\ell$-homogeneous sequences, with the extra property (needed for induction)  that the cardinality of the sequence that ends in the last point of the interval is strictly greater than $|T_{\ell}| + 1$. 

We prove that for any $T_{\ell-1}$-homogeneous sequence, there exists a tiling of its Cartesian product with some interval by $V_\ell$-homogeneous paths. From these tilings, we build the tiling of a rectangle of width $L+1$. By {\it deleting} the last point, we obtain a tiling of $[0,L-1]$ by $T_{\ell-1}$-homogeneous sequences. We can again turn it into a tiling of a rectangle of width $L$ by $V_\ell$-homogeneous paths. Now we can make a tiling of a rectangle of sufficiently large width $d_\ell$ by $V_\ell$-homogeneous paths using Proposition~\ref{lem0}. The ''lexicographical'' map from this rectangle to the interval maps any $V_\ell$-homogeneous path to $T_\ell$-homogeneous sequence. 

To complete the proof, we need to pass from homogeneous sequences to an actual tiling. We prove that a Cartesian product of a relatively short $T_{s+p-1}$-homogeneous sequence and some interval may be tiled by {\it `actual' paths} of type $V_{s+p}$. Concretely, we consider a tiling of an interval $[0,L]$ by $T_{s+p-1}$-homogeneous sequences obtained in the previous part. It gives us a tiling of a rectangle of width $L+1$ by paths of type $V_{s+p}$. We can delete the last point of $[0,L]$ and get a tiling of $[0,L-1]$ by $T_{s+p-1}$-homogeneous sequences, because the sequence ending in point $L$ was ''long''. From this, we obtain a tiling of a rectangle of width $L$ by paths of type $V_{s+p}$. Finally, we can again use Proposition~\ref{lem0} and construct one rectangle of width $d_{s+p}$ and send it into an interval by the ''lexicographical'' map.


\subsection{Proof of the theorem}
Now let us switch to the proof. It is based on two lemmas. 

\begin{Lemma}\label{lem1} For each $\ell,$ $2\leq \ell \leq s$,  there exists a tiling of some interval by tiles with the gap set $T_\ell,$ with the following additional property: for each of the tiles that end in one of the last $d_1$ points of this interval, the first $k_1 - \sum_{i = 3}^{\ell} k_i$ gap lengths are equal to $d_1$.
\end{Lemma}

\begin{proof}[Proof of Lemma~\ref{lem1}] As the reader may expect, the proof uses induction on $\ell.$ 
\vskip+0.1cm
{\underline {Case 1. $\ell = 2.$}}
\vskip+0.1cm

Essentially, the proof of base case is a slightly modified proof of the main theorem from \cite{Naka}. By Proposition~\ref{propnaka}, the rectangle $[0,k_1+k_2-1] \times [0,f_{k_1,k_2}(k_1+k_2)-1]$ can be tiled by paths of type $\{e_1^{(k_1)},e_2^{(k_2)}\}$. Besides, $[0,k_1+k_2] \times [0,k_2]$  can be tiled by paths of the same type, and the path that end in the point $(k_1+k_2,k_2)$ has the first $k_1$ gaps equal to $e_1$. Indeed, paths in such a tiling are as follows (cf. Fig.~\ref{fig1}): 
\begin{align*}
W_0 &= (0, 0)\to (0,k_2)\to (k_1,k_2),\\
W_1 &= (1,0)\to(1,k_2-1)\to (k_1+1,k_2-1)\to (k_1+1,k_2),\\
&\ldots \\
W_i &= (i,0)\to (i,k_2-i)\to  (k_1+i,k_2-i)\to (k_1+i,k_2)\\
&\ldots\\ 
W_{k_2} &= (k_2,0)\to (k_1+k_2, 0)\to (k_1+k_2, k_2).
\end{align*}

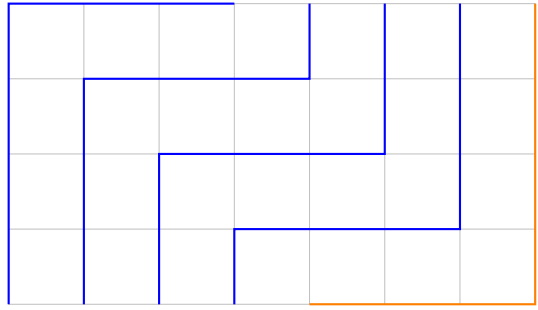
\begin{figure}[h]
\begin{tikzpicture}
    \draw[step=1cm, very thin, gray!50] (0,0) grid (7,4);
    \draw[orange, thick] (4,0)--(7,0)--(7,4);
    \foreach \x in {0,1,2,3} \draw[blue, thick] (\x,0)--(\x, 4-\x)--(\x+3,4-\x)--(\x+3,4);
\end{tikzpicture}
\caption{Tiling of a rectangle $[0,k_1+k_2]\times [0,k_2]$ by paths of type $\{e_1^{(k_1)},e_2^{(k_2)}\}$}\label{fig1}
\end{figure}

Let $a \geq (k_1+k_2)^2 + 1$ be a positive integer. By Proposition~\ref{lem0}, there exist non-negative integers $c_1$ and $c_2$ and positive integers $b_1$ and $b_2$, such that $a = b_1(k_1+k_2+1) + c_1(k_1+k_2)$ and $a + 1 = b_2(k_1+k_2+1) + c_2(k_1+k_2).$ Then for each positive integer $h$  rectangles $[0, a-1] \times [0, h-1]$ and $[0, a] \times [0, h-1]$ can be partitioned into copies of $[0,k_1+k_2 - 1] \times [0,h-1]$ (first type) and $[0,k_1+k_2] \times [0,h-1]$ (second type),  with the number of rectangles of the second type strictly positive in both cases. We work with partitions where rectangles of the first type are placed to the left and rectangles of the second type are placed to the right.

Denote $h = \text{lcm}(k_2+1,f_{k_1,k_2}(k_1+k_2))$. Then $[0,k_1+k_2 - 1] \times [0,h-1]$ can be partitioned into (disjoint) copies of $[0,k_1+k_2-1] \times [0,f_{k_1,k_2}(k_1+k_2)-1]$ and $[0,k_1+k_2] \times [0,h-1]$ can be partitioned into  copies of $[0,k_1+k_2] \times [0,k_2]$. Combining these two partitions, we obtain partitions of $[0, a-1] \times [0, h-1]$ and $[0, a] \times [0, h-1]$ with a copy of $[0,k_1+k_2] \times [0,k_2]$ in the top right corner (cf. Fig.~\ref{fig2}).

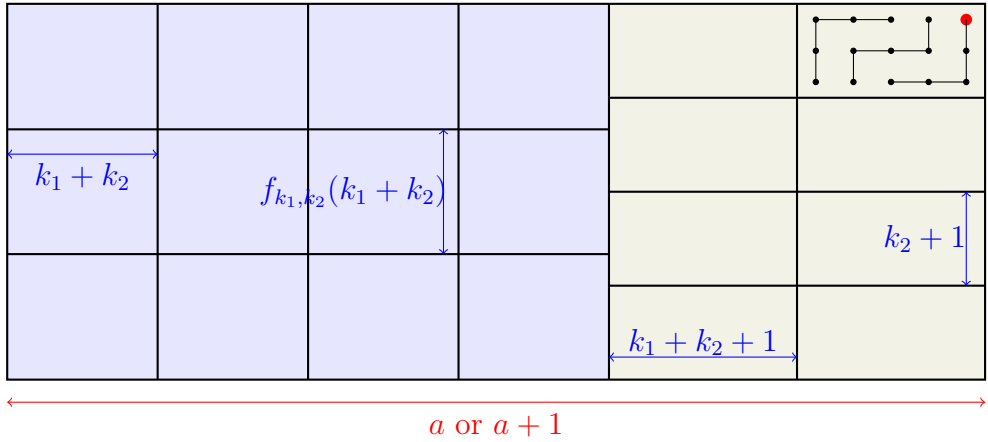
\begin{figure}[h]
\begin{tikzpicture}
    \filldraw[blue!10] (0,0) rectangle (8,5);
    \filldraw[green!50!red!10] (8,0) rectangle (13,5);
    \draw[black, thick] (0,0) rectangle (13,5);
    \draw[black, thick] (2,0) -- (2,5);
    \draw[black, thick] (4,0) -- (4,5);
    \draw[black, thick] (6,0) -- (6,5);
    \draw[black, thick] (8,0) -- (8,5);
    \draw[black, thick] (10.5,0) -- (10.5,5);
    \draw[black, thick] (0,1.67) -- (8,1.67);
    \draw[black, thick] (0,3.33) -- (8,3.33);
    \draw[black, thick] (8,1.25) -- (13,1.25);
    \draw[black, thick] (8,2.5) -- (13,2.5);
    \draw[black, thick] (8,3.75) -- (13,3.75);
    \draw[red, <->] (0,-0.3)--(13,-0.3);
    \draw[red] (6.5, -0.6) node{$a$ or $a+1$};

    \draw[<->,blue,thin] (0,3)--(2,3);
    \draw[blue] (1,2.7) node{$k_1+k_2$};

    \draw[<->,blue,thin] (5.8,1.67)--(5.8,3.33);
    \draw[blue] (4.6,2.5) node{$f_{k_1,k_2}(k_1+k_2)$};

    \draw[<->,blue,thin] (8,0.3)--(10.5,0.3);
    \draw[blue] (9.25,0.5) node{$k_1+k_2+1$};

    \draw[<->,blue,thin] (12.75,1.25)--(12.75,2.5);
    \draw[blue] (12.2,1.875) node{$k_2+1$};

    \foreach \x in {10.75, 11.25,...,12.75} \filldraw[black] (\x, 3.96) circle[radius=1pt]; 
    \foreach \x in {10.75, 11.25,...,12.75} \filldraw[black] (\x, 4.375) circle[radius=1pt];
    \foreach \x in {10.75, 11.25,...,12.25} \filldraw[black] (\x, 4.79) circle[radius=1pt];

    \filldraw[red] (12.75, 4.79) circle[radius=2pt];

    \draw[black] (11.75,3.96)--(12.75,3.96)--(12.75,4.79);
    \draw[black] (11.25,3.96)--(11.25,4.375)--(12.25,4.375)--(12.25,4.79);
    \draw[black] (10.75,3.96)--(10.75,4.79)--(11.75,4.79);

\end{tikzpicture}
\caption{Partition of  $[0,a-1]\times[0,h-1]$ or $[0,a]\times[0,h-1]$ into rectangles $[0,k_1+k_2-1]\times[0,f_{k_1,k_2}(k_1+k_2)-1]$ and $[0,k_1+k_2]\times[0,k_2]$. In the upper right corner, we have a translate of a rectangle $[0,k_1+k_2]\times[0,k_2]$ with a partition into paths as on Fig.~\ref{fig1}}
\label{fig2}
\end{figure}

Then, using the same trick as in \cite{Naka}, we can construct a tiling of a rectangle of width $d_2 = a d_1 + t$ ($0 \leq t<d_1$) and height $h$ by paths of type $V_2 = \{(d_1 e_1)^{(k_1)},e_2^{(k_2)}\}$. We obtain it from tilings of $[0, a-1] \times [0, h-1]$ and $[0, a] \times [0, h-1]$ constructed above. In addition, this tiling shall have  the following property: paths that end in points $(d_2-1, h-1), (d_2-2, h-1), ..., (d_2-d_1, h-1)$, have the first $k_1$ gaps equal to $d_1 e_1.$ In order to do so, let us divide an interval as follows:  $$[0,d_2-1] = [0, a d_1 + t - 1] = \bigsqcup_{i = 0}^{t-1} \big\{d_1 x + i\, :\, x\in [0,a]\big\} \sqcup \bigsqcup_{i = t}^{d_1-1} \big\{d_1 x + i\, :\, x\in [0,a-1]\big\}.$$ 
These sets are disjoint, because points from distinct sets have distinct residues modulo $d_1$. Computing the number of points in the set on the right-hand side, as well its minimal and maximal element, we see that the sets are indeed equal.
From here, we obtain a partition of our rectangle into disjoint sets: 
\begin{align*}[0, d_2-1]\times [0,h-1] = &\bigsqcup_{i = 0}^{t-1}\big\{(d_1x+i,y)\,:x\in[0,a], y\in [0,h-1]\big\} \\ \sqcup &\bigsqcup_{i = t}^{d_1-1}\{(d_1x+i,y)\,:x\in[0,a-1], y\in [0,h-1]\big\}.
\end{align*} Each set on the right-hand side can easily be tiled by paths of type $V_2 = \{(d_1 e_1)^{(k_1)},e_2^{(k_2)}\}.$ Indeed, it suffices to do a homothety with ratio $d_1$ along the $x$-axis of the previously constructed tilings of $[0, a-1] \times [0, h-1]$ and $[0, a] \times [0, h-1]$ by paths of type $\{e_1^{(k_1)},e_2^{(k_2)}\}$. 

Note that points $(d_2-1, h-1), (d_2-2, h-1), ..., (d_2-d_1, h-1)$ are the top right points of the sets from the above decomposition, and thus the paths that end in them correspond to the paths that end in the top right corner of $[0, a-1] \times [0, h-1]$ or $[0, a] \times [0, h-1]$, so their first $k_1$ gaps equal to $d_1 e_1.$

Finally, consider a map  $[0, d_2-1]\times [0,h-1]\to [0,hd_2-1],$ which sends a point $(a,b)$ to $a+bd_2.$ 
This map sends each path of type $V_2$ to a tile with gap set $T_2 = \{ d_1^{(k_1)}, d_2^{(k_2)}\}$. Recall that $(d_2-1, h-1), (d_2-2, h-1), ..., (d_2-d_1, h-1)$ are the end points of paths with first $k_1$ gaps equal to $d_1 e_1.$ These points are mapped to the last $d_1$ points of the interval. Consequently, in the constructed tiling, each of the tiles that end in one of the last $d_1$  points has first $k_1$ gap lengths equal to $d_1.$ 
The base case is  proved.
\vskip+0.1cm
{\underline {Case 2. $\ell \geq 3.$}}\vskip+0.1cm

Now assume that $\ell>2$ and we have verified the statement  for $\ell-1.$ The induction hypothesis gives that we can tile  $[0,L]$ using tiles with gap set $T_{\ell - 1}$.  In addition, the first $k_1 - \sum_{i = 3}^{\ell-1}k_i$ gaps of tiles that end in $L, L-1, ..., L-d_1+1$ are equal to $d_1$. 

Let $0\leq t\leq d_1-1.$ Denote $n = \sum_{i=1}^{\ell-1} k_i$ and $k = k_1 - \sum_{i = 3}^{\ell-1}k_i$. Define the tile $(v_0^t, ..., v_n^t)$ with the last point $v_n^t = L - t.$ By the induction hypothesis, $v_1^t - v_0^t = v_2^t - v_1^t = ... = v_{k}^t - v_{k-1}^t = d_1.$ Denote $v_{n+j}^t = L - t + d_1 j$ for each $j=1,...,k_\ell.$ 

Consider the following sets in $\mathbb{Z}^2$: \begin{align*}
W_{1,i}^t &= \{v_{k_\ell-i}^t\} \times [0,i],\\ 
W_{2,i}^t &= \{v_{k_\ell-i}^t, v_{k_\ell-i+1}^t, ..., v_{k_\ell-i+n}^t\}\times \{i\},\\ 
W_{3,i}^t &= \{v_{k_\ell-i+n}^t\} \times [i,k_\ell],\end{align*} where $i = 0, ..., k_\ell.$ Put $W_i^t = W_{1,i}^t \cup W_{2,i}^t \cup W_{3,i}^t.$ 

The set $W_{1,i}^t$ forms a vertical path containing $i$ gaps of length 1. The last point of $W_{1,i}^t$ coincides with  the first point of $W_{2,i}^t.$ $W_{2,i}^t$ consists of $n = \sum_{j=1}^{\ell-1} k_j$ horizontal gaps. Its last point coincides with the first point of the set $W_{3,i}^t,$ and $W_{3,i}^t$ consists of $k_\ell - i$ unit vertical gaps. 

\vskip+0.1cm
\begin{claim}
   $W_i^t$ is a path of type $V_\ell.$ 
\end{claim}
\begin{proof}
    It suffices to check that $W_{2,i}^t$ is a translation of a tile with the gap set $T_{\ell-1}$ or, equivalently, that $(v_{k_\ell-i}^t, v_{k_\ell-i+1}^t, ..., v_{k_\ell - i+n}^t)$ is such a tile. Let us prove that this is the case. Gaps $v_{k_\ell-i+1}^t-v_{k_\ell-i}^t,$ $v_{k_\ell-i+2}^t-v_{k_\ell-i+1}^t,...,$ $ v_{n}^t-v_{n-1}^t$ are common for our tile and the tile $(v_0^t, ..., v_n^t)$. The gap set of $(v_{k_\ell-i}^t,$ $ v_{k_\ell-i+1}^t,$ $ ..., v_{k_\ell - i+n}^t)$ includes, apart from these gaps, gaps $v_{n+1}^t-v_n^t,..., v_{k_\ell - i+n}^t - v_{k_\ell - i+n-1}^t,$ but they are all equal to $d_1$ by definition of $v_j^t$ with $j>n.$ Gaps in $(v_0^t, ..., v_n^t)$ not shared with our tile are $v_1^t - v_0^t, v_2^t - v_1^t, ..., v_{k_\ell-i}^t - v_{k_\ell-i-1}^t.$ Since the condition of the theorem implies $k_\ell < k$, all these ``extra'' gaps of $(v_0^t, ..., v_n^t)$ belong to the first $k$ ones and thus are equal to $d_1$ by the induction hypothesis. Consequently, gap sets of $(v_0^t, ..., v_n^t)$ and $(v_{k_\ell-i}^t, v_{k_\ell-i+1}^t, ..., v_{k_\ell - i+n}^t)$ contain only common gaps and gaps of length $d_1.$ Hence, they are the same.
\end{proof}

Further, paths $W_i^t$ tile $\{v_0^t,...,v_{n+k_\ell}^t\} \times [0,k_\ell]$ similarly to Fig.~\ref{fig1}. Indeed, the subset $$\{(v_x^t,y)\,:\, x+y \leq k_\ell, x\geq 0, y\geq 0\}$$ is tiled by $W_{1,i}^t$: all $W_{1,i}^t$ belong to this set and each point $(v_x^t, y)$ of this set is covered by $W_{1,k_\ell - x}^t$ and this is the only set that covers this point. Similarly, $W_{2,i}^t$ tile the subset $$\{(v_x^t,y) \,:\,  k_\ell\leq x+y \leq k_\ell+n, x\geq 0, 0\leq y\leq k_\ell\}$$ and $W_{3,i}^t$ tile $$\{(v_x^t,y)\,:\,  x+y \geq k_\ell+n, 0\leq x\leq k_\ell + n, 0\leq y\leq k_\ell\}.$$

Finally, note that first $k - k_\ell$ gaps of $W_0^t = \{v_{k_\ell}^t, ..., v_{k_\ell + n}^t\}\times \{0\} \cup \{v_{k_\ell+n}^t\}\times [0,k_\ell]$, that ends in $(L-t+k_\ell d_1, k_\ell)$, are $(v_{k_\ell+1}^t - v_{k_\ell}^t)e_1, ..., (v_{k}^t - v_{k - 1}^t)e_1$. Since $v_{k_\ell+1}^t - v_{k_\ell}^t, ..., v_{k}^t - v_{k - 1}^t$ belong to the first $k$ gaps of the tile $(v_0^t,...,v_n^t)$, all their lengths are equal to $d_1$ by the induction hypothesis. Hence, the first $k - k_\ell = k_1 - \sum_{i=3}^{\ell}k_i$ gaps of $W_0^t$ are equal to $d_1 e_1.$

Now let $(v_0,...,v_n)$ be any tile. According to Proposition~\ref{propnaka}, there exists a positive integer $f_{n,k_\ell}(n+1)$ such that paths of type $\{e_1^{(n)}, e_2^{(k_\ell)}\}$ tile $[0,n]\times [0,f_{n,k_\ell}(n+1)-1]$. Consider the bijection \begin{align*}
    \varphi:\ & [0,n]\times [0,f_{n,k_\ell}(n+1)-1]\rightarrow \{v_0,...,v_n\}\times [0,f_{n,k_\ell}(n+1)-1],\\ \varphi:\ & (x,y) \to (v_x,y).\end{align*}
Note that $\varphi$ maps a path of type $\{e_1^{(n)}, e_2^{(k_\ell)}\}$ into a path 
of type $V_\ell$. Hence, $\varphi$ maps the tiling of $[0,n]\times [0,f_{n,k_\ell}(n+1)-1]$ by paths of type $\{e_1^{(n)}, e_2^{(k_\ell)}\}$ into a tiling of $\{v_0,...,v_n\}\times [0,f_{n,k_\ell}(n+1)-1]$ by paths of the type $V_\ell$. 

Finally, consider a tile $(v_0^{d_1-1},...,v_n^{d_1-1}).$ Recall that $v_n^{d_1-1} = L - d_1 + 1.$ By the induction hypothesis, $v_1^{d_1-1}-v_0^{d_1-1} = d_1.$ By Proposition~\ref{propnaka}, there exists an integer $f_{n,k_\ell}(n+2)$ such that paths of type $\{e_1^{(n)}, e_2^{(k_\ell)}\}$ tile $[0,n+1]\times [0,f_{n,k_\ell}(n+2)-1]$. Denote $v_{n+1}^{d_1-1} = L+1$. Observe the bijection \begin{align*}\varphi:\ &[0,n+1]\times [0,f_{n,k_\ell}(n+2)-1]\rightarrow \{v_0^{d_1-1},...,v_n^{d_1-1}, v_{n+1}^{d_1-1}\}\times [0,f_{n,k_\ell}(n+2)-1], \\  
\varphi:\ & (x,y) \to (v_x^{d_1-1},y).\end{align*} It maps a path of type $\{e_1^{(n)}, e_2^{(k_\ell)}\}$ into a path having $k_\ell$ gaps equal to $e_2$ and either gaps $(v_1^{d_1-1} - v_0^{d_1-1})e_1, ..., (v_n^{d_1-1} - v_{n-1}^{d_1-1})e_1$ or gaps $(v_2^{d_1-1} - v_1^{d_1-1})e_1, ..., (v_{n+1}^{d_1-1} - v_{n}^{d_1-1})e_1$. But since $v_1^{d_1-1}-v_0^{d_1-1} = d_1 = v_{n+1}^{d_1-1}-v_n^{d_1-1}$, these multisets of horizontal gaps are the same and equal to $T_{\ell-1}$ and thus all paths in the image of $\varphi$ are of type $V_\ell.$

Therefore, we obtain tilings of the following sets by paths of type $V_\ell$:
\begin{itemize}
    \item $\{v_0,...,v_n\}\times [0,f_{n,k_\ell}(n+1)-1],$ where $(v_0,...,v_n)$ is any tile used in the tiling of $[0,L]$ from the induction hypothesis;
    \item $\{v_0^{d_1-1},...,v_n^{d_1-1}, L+1\}\times [0,f_{n,k_\ell}(n+2)-1],$ where $(v_0^{d_1-1},...,v_n^{d_1-1})$ is the tile used in the tiling of $[0,L]$ from the induction hypothesis and ends in $v_n^{d_1-1} = L-d_1+1$;
    \item $\{v_0^t,...,v_{n}^t, L-t+d_1,..., L-t+k_\ell d_1\} \times [0,k_\ell]$, where $0\leq t\leq d_1-1$ and $(v_0^t,...,v_{n}^t)$ is the tile used in the tiling of $[0,L]$ 
 and ends in $v_n^t = L-t$.
\end{itemize}    
In addition, in the tiles of the third type the point $(L-t+k_\ell d_1,k_\ell)$ belongs to the path with first $k_1 - \sum_{i=3}^\ell k_i$ gaps equal to $d_1 e_1.$

Some of these tilings are shown in Fig.~\ref{fig3}. In the figure points of $[0,L]$ that belong to the same tile have the same colour. Points of $[L+1,L+k_\ell d_1]$ are highlighted with empty red circles and have the same colour as the tile that they are added to in the inductive step. The red paths illustrate the tiling of $\{v_0,...,v_n\}\times [0,f_{n,k_\ell}(n+1)-1],$ for the corresponding red tile on the line. The blue paths illustrate 
the tiling of $\{v_0^0,...,v_{n}^0, L+d_1,..., L+k_\ell d_1\} \times [0,k_\ell]$, where $(v_0^0,...,v_{n}^0)$ is a tile ending in $v_n^0 = L$.

\begin{figure}[h]
\begin{tikzpicture}[scale = 0.8]
    \draw[step=0.5cm, very thin, gray!50] (0,0) grid (14.95,3);
    \draw[step=0.5cm, very thin, gray!50!red!50] (14.95,0) grid (17.5,3);

    \filldraw[red] (17.5,0) circle[radius=3pt]
                    (17,0) circle[radius=3pt]
                    (16.5,0) circle[radius=3pt]
                    (16,0) circle[radius=3pt]
                    (15.5,0) circle[radius=3pt]
                    (15,0) circle[radius=3pt];

    \draw[red, thick] (0,0)--(12,0)--(12,1);
    \draw[red, thick] (0,0.5)--(10.5,0.5)--(10.5,1.5)--(12,1.5);
    \draw[red, thick] (0,1)--(9,1)--(9,2)--(12,2);
    \draw[red, thick] (0,1.5)--(7.5,1.5)--(7.5,2.5)--(12,2.5);
    \draw[red, thick] (0,2)--(0,3)--(12,3);
    
    \foreach \y in {0,0.5,...,3}\filldraw[red] (0,\y) circle[radius=2pt]
                    (7.5,\y) circle[radius=2pt]
                    (9,\y) circle[radius=2pt]
                    (10.5,\y) circle[radius=2pt]
                    (12,\y) circle[radius=2pt];
    \filldraw[orange] (0.5,0) circle[radius=2pt]
                    (8,0) circle[radius=2pt]
                    (9.5,0) circle[radius=2pt]
                    (11,0) circle[radius=2pt]
                    (12.5,0) circle[radius=2pt];
    \filldraw[orange!70!yellow!50] (1,0) circle[radius=2pt]
                    (8.5,0) circle[radius=2pt]
                    (10,0) circle[radius=2pt]
                    (11.5,0) circle[radius=2pt]
                    (13,0) circle[radius=2pt];
    \filldraw[green] (1.5,0) circle[radius=2pt]
                    (3,0) circle[radius=2pt]
                    (4.5,0) circle[radius=2pt]
                    (6,0) circle[radius=2pt]
                    (13.5,0) circle[radius=2pt]
                    (15,0) circle[radius=2pt]
                    (16.5,0) circle[radius=2pt];
    \filldraw[green!50!blue!50] (2,0) circle[radius=2pt]
                    (3.5,0) circle[radius=2pt]
                    (5,0) circle[radius=2pt]
                    (6.5,0) circle[radius=2pt]
                    (14,0) circle[radius=2pt]
                    (15.5,0) circle[radius=2pt]
                    (17,0) circle[radius=2pt];
    \foreach \y in {0,0.5,1}\filldraw[blue] (2.5,\y) circle[radius=2pt]
                    (4,\y) circle[radius=2pt]
                    (5.5,\y) circle[radius=2pt]
                    (7,\y) circle[radius=2pt]
                    (14.5,\y) circle[radius=2pt]
                    (16,\y) circle[radius=2pt]
                    (17.5,\y) circle[radius=2pt];

    \draw[blue, thick] (5.5,0)--(17.5,0)--(17.5,1);
    \draw[blue, thick] (4,0)--(4,0.5)--(16,0.5)--(16,1);
    \draw[blue, thick] (2.5,0)--(2.5,1)--(14.5,1);

    \draw[blue] (14.5,-0.4) node{$v_n^0 = L$};
    \draw[blue] (2.5,-0.4) node{$v_0^0$};
    \draw[blue] (4,-0.4) node{$v_1^0$};
    \draw[blue] (17.5,-0.4) node{$L + k_\ell d_1$};

    \draw[red] (0,-0.4) node{$v_0$};
    \draw[red] (12,-0.4) node{$v_n$};
    
\end{tikzpicture}
\caption{A part of tiling of $[0, L+k_\ell d_1]\times [0,h-1]$. } \label{fig3}
\end{figure}
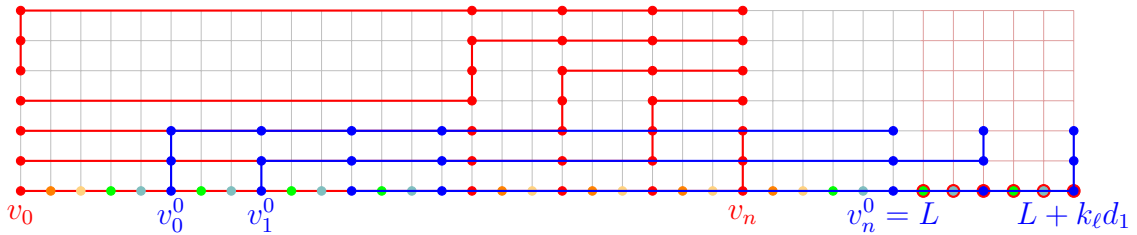
Denote $h = \text{lcm}(k_\ell+1, f_{n,k_\ell}(n+1), f_{n,k_\ell}(n+2))$. Denote by $\mathfrak{F}$ the family of all tiles  used in the tiling of $[0,L]$ from the induction hypothesis. Denote by $\mathfrak{F}'$ the family of all such tiles, except the one ending in $L-d_1+1,$ and by $\mathfrak{F}''$  the family of all such tiles, except the ones ending in $L, L-1, ..., L - d_1 + 1.$ Consider the following partitions: 
{\small \begin{align*}[0,L] \times [0,h-1] = &\bigsqcup_{\{v_0,...,v_n\}\in \mathfrak{F}} (\{v_0,...,v_n\}\times[0,h-1]),\\
[0,L+1] \times [0,h-1] = &\bigsqcup_{\{v_0,...,v_n\}\in \mathfrak{F}'} (\{v_0,...,v_n\}\times[0,h-1]) \sqcup \{v_0^{d_1-1}, ..., v_n^{d_1-1}, L+1\}\times [0,h-1],\\
[0,L + k_\ell d_1] \times [0,h-1] = &\bigsqcup_{\{v_0,...,v_n\}\in \mathfrak{F}''} (\{v_0,...,v_n\}\times[0,h-1])  \\
&\ \sqcup\  \bigsqcup_{t = 0}^{d-1} \{v_0^t, ..., v_n^t, L-t+d_1,...,L-t+k_\ell d_1\}\times [0,h-1],\\
\{v_0,...,v_n\}\times[0,h-1] = &\bigsqcup_{j=0}^{h/f_{n,k_\ell}(n+1)-1}(\{v_0,...,v_n\}\times[0,f_{n,k_\ell}(n+1)-1] + (0,f_{n,k_\ell}(n+1)j));
\end{align*}
\begin{multline*}
\{v_0^{d_1-1},...,v_n^{d_1-1},L+1\}\times[0,h-1] =\\ \bigsqcup_{j=0}^{h/f_{n,k_\ell}(n+2)-1}(\{v_0^{d_1-1},...,v_n^{d_1-1},L+1\}\times[0,f_{n,k_\ell}(n+2)-1] + (0,f_{n,k_\ell}(n+2)j));    
\end{multline*}
\begin{multline*}
    \{v_0^t,...,v_n^t, L-t+d_1, ..., L-t+k_\ell d_1\}\times[0,h-1] =\\  \bigsqcup_{j=0}^{h/(k_\ell+1)-1}(\{v_0^t,...,v_n^t, L-t+d_1, ..., L-t+k_\ell d_1\}\times[0,k_\ell] + (0,(k_\ell+1) j)).
\end{multline*}
}

The above illustrates that rectangles $[0,L] \times [0,h-1]$, $[0,L+1] \times [0,h-1]$, and $[0,L + k_\ell d_1] \times [0,h-1]$ are partitioned into translations of sets tiled by paths of type $V_\ell$. 

Set $g_\ell(d_1,..., d_{\ell-1}, k_1,...,k_{\ell-1}, k_\ell) = L(L+1) + (L+k_\ell d_1+1).$ Then by Proposition~\ref{lem0} for any $d_\ell \geq g_\ell(d_1,..., d_{\ell-1}, k_1,..., k_\ell)$ the rectangle $[0, d_\ell - (L+k_\ell d_1+1) - 1]\times [0,h-1]$ can be split into  translations of $[0,L] \times [0,h-1]$ and $[0,L+1] \times [0,h-1]$. Adding a rectangle $[0,L + k_\ell d_1] \times [0,h-1]$ on the right, we get a tiling by paths of type $V_\ell$ of the rectangle $[0, d_\ell - 1]\times [0,h-1]$. Points $(d_\ell-1, h-1), (d_\ell-2, h-1), ..., (d_\ell - d_1, h-1)$ are images of $(L + k_\ell d_1, h-1), (L + k_\ell d_1 - 1, h-1), ..., (L + k_\ell d_1 - d_1 + 1, h-1)$ under translation. These points, in turn, are obtained from points $(L-t+k_\ell d_1,k_\ell)$ from tilings of $\{v_0^t,...,v_{n}^t, L-t+d_1,..., L-t+k_\ell d_1\} \times [0,k_\ell]$. Hence, as shown above, paths that end in these points have first $k_1 - \sum_{i=3}^\ell k_i$ gaps equal to $d_1 e_1.$

Now let us consider a map similar to the one used in base case, $(a,b)  \mapsto a + d_\ell b.$
It maps paths of type $V_l$ to tiles with gap set $T_\ell.$ The last $d_1$ points of the obtained interval are the images of $(d_1-1, h-1), (d_\ell-2, h-1), ..., (d_\ell - d_1, h-1)$. These points are the endpoints of paths with first $k_1 - \sum_{i=3}^\ell k_i$ gaps equal to $d_1 e_1,$ and, consequently, the last $d_1$ points of the interval are end-points of tiles with first $k_1 - \sum_{i=3}^\ell k_i$ gap lengths equal to $d_1$. Therefore, the inductive step is also verified.
\end{proof}

In order to state the second lemma, we need some notation. Let $T$ be some gap set. We call a sequence of integers $T$-homogeneous, if any $|T|+1$ consecutive points in this sequence form a tile with gap set $T$. Similarly, for a multiset $V$ of vectors in $\Z^2$ we call a path $V$-homogeneous if any $|V|+1$ consecutive points in this path form a path of type $V$.

\begin{Lemma}
    For each $\ell = 0,...,p-1,$  there exists a tiling of an interval by $T_{s+\ell}$-homogeneous sequences, and none of these sequences contain more than $|T_{s+\ell}| + \sum_{i=1}^{\ell}k_{s+i} + 2$ points and the cardinality of the sequence that ends in the last point of the interval is strictly greater than $|T_{s+\ell}| + 1$.
\label{lem2}
\end{Lemma}

\begin{proof} We again use induction on $\ell$.
    \vskip+0.1cm
{\underline {Case 1. $\ell = 0.$}}
\vskip+0.1cm
From Lemma~\ref{lem1} it is proved that there exists a tiling of some interval $[0,L]$ by tiles with gap set $T_s,$ moreover, for a tile $(v_0,...,v_n)$ that ends in $L-d_1+1$ we have $v_1-v_0 = d_1.$ Consider the sequence $(v_0,...,v_n, L+1).$ We have $v_1 - v_0 = d_1 = (L+1) - v_n,$ and $(v_0,...,v_n)$ is a tile with gap set $T_s$. Hence, the tile $(v_1,...,v_n, L+1)$ also has gap set $T_s$, and thus the sequence $(v_0,...,v_n, L+1)$ is $T_s$-homogeneous and it has cardinality $|T_s|+2$. Consequently, this sequence and all tiles, except for $(v_0,...,v_n)$, form the desired tiling of $[0,L+1].$
\vskip+0.1cm
{\underline {Case 2. $\ell \geq 1.$}}
\vskip+0.1cm
Using the induction hypothesis, consider a tiling of some interval $[0,L]$ by $T_{s+\ell-1}$-homogeneous sequences.
Let $(v_0,...,v_m)$ be a $T_{s+\ell-1}$-homogeneous sequence from this tiling with $m > |T_{s+\ell-1}|.$ We shall prove that there exists a tiling of $\{v_0,...,v_m\} \times [0,|T_{s+\ell-1}| + k_{s+\ell}]$ by $V_{s+\ell}$-homogeneous paths. 

As before, denote $n = |T_{s+\ell-1}|$. Consider the rectangle $[0,m]\times [0, m + k_{s+\ell}].$ Take the following partition of this rectangle into paths (cf. Fig~\ref{fig4}). Draw the lines $\{x + y = m + t(n+k_{s+\ell})\,:\, t\in \mathbb{Z}\}$ and $\{x + y = m + tn + (t+1)k_{s+\ell}\,:\, t\in \mathbb{Z}\}.$
They divide the plane into alternating stripes of widths $n$ and $k_{s+\ell}$.

Let all points lying on one horizontal line inside a stripe of width $n$ (including the boundary points) belong to the same path, and, similarly, all points on one vertical line inside a stripe of width $k_{s+\ell}$ also belong to the same path. This way we have a partition of $[0,m]\times [0, m + k_{s+\ell}]$ into paths, where each of these paths is a piece of an infinite path with alternating parts of $n$ horizontal steps and $k_{s+\ell}$ vertical steps. In particular, all these paths have the following property: among any $n+k_{s+\ell}$ consecutive unit gaps  exactly $k_{s+\ell}$ are vertical. 

\vskip+0.1cm
\begin{claim} All paths in this partition, except for $P=(m-n,0)\to (m,0)\to  (m,k_{s+\ell})$ and $P' = (0,m)\to(0,m+k_{s+\ell})\to (n,m+k_{s+\ell}),$ have strictly more than $n+k_{s+\ell}$ gaps. In particular, it is true for a path that ends in the point $(m,m+k_{s+\ell}).$
\end{claim}
\begin{proof}
Indeed, the first point of any path is $(x, y)$ with either $x = 0$ or $y = 0,$ and the endpoint is $(x,y)$  with either $x = m$ or $y = m+k_{s+\ell}.$ Besides, other paths cannot begin or end in points that belong to $P$ or $P'$. 
Hence, all these paths begin in $(x,y)$ with $x = 0$ and $y \leq m-1$ or $y = 0$ and $x \leq m-n-1$, and end in $(x,y)$ with $x = m$ and $y \geq k_{s+\ell}+1$ or $y = m+k_{s+\ell}$ and $x \geq n+1.$ Consequently, for any path either the difference in the $x$-coordinates of the last and first points in that path is  strictly greater than $n$ or the difference in their $y$-coordinates is strictly greater than $k_{s+\ell}$. On the other hand, any path consists of alternating horizontal pieces of length $n$ and vertical pieces of length $k_{s+\ell}$. This implies that if such a path had length at most $n+k_{s+\ell}$ then the differences between the $x$- and $y$-coordinates of the end and start points would be less than or equal to $n$ and $k_{s+\ell}$, respectively. 
\end{proof}

\begin{figure}

\begin{tikzpicture}

\fill[blue!5] (0,0)--(5.5,0)--(5.5,6.5)--(0,6.5);

\fill[orange!30!yellow!20] (0,5.5)--(5.5,0)--(5.5,1)--(0,6.5);
\fill[orange!30!yellow!20] (3,6.5)--(5.5,4)--(5.5,5)--(4,6.5);
\fill[orange!30!yellow!20] (0,1.5)--(1.5,0)--(2.5,0)--(0,2.5);

\draw[orange!70,dashed] (0,5.5)--(5.5,0);
\draw[orange!70,dashed] (5.5,1)--(0,6.5);
\draw[orange!70,dashed] (3,6.5)--(5.5,4);
\draw[orange!70,dashed] (5.5,5)--(4,6.5);
\draw[orange!70,dashed] (0,1.5)--(1.5,0);
\draw[orange!70,dashed] (2.5,0)--(0,2.5);

\draw[step=0.5cm, very thin, gray!50] (0,0) grid (5.5,6.5);
\foreach \x in {0,0.5,...,2.5} \draw[thick, blue] (\x,2.5-\x) -- (\x+3, 2.5-\x);
\foreach \y in {3,3.5,4,4.5,5} \draw[thick, blue] (0,\y)--(5.5-\y,\y);
\foreach \x in {0,0.5,...,5.5} \draw[thick, blue] (\x,5.5-\x)--(\x,6.5-\x);
\foreach \x in {0,0.5,...,2.5} \draw[thick, blue] (\x, 6.5-\x)--(\x+3, 6.5-\x);
\foreach \x in {3,3.5,...,5} \draw[thick, blue] (\x, 6.5-\x)--(5.5, 6.5-\x);
\foreach \x in {4,4.5,5,5.5} \draw[thick, blue] (\x, 9.5-\x)--(\x, 10.5-\x);
\draw[thick, blue] (3.5,6)--(3.5,6.5);
\foreach \y in {6.5,6,5.5} \draw[thick, blue] (10.5-\y,\y)--(5.5,\y);
\foreach \x in {0,0.5,1,1.5} \draw[thick, blue] (\x, 1.5-\x)--(\x,2.5-\x);
\draw[thick, blue] (2,0)--(2,0.5);
\foreach \y in {0,0.5,1} \draw[thick, blue] (0,\y)--(1.5-\y,\y);

\draw[<->,blue,thin] (2.5,-0.2)--(5.5,-0.2);
\draw[blue] (4,-0.4) node{$n$};

\draw[<->,blue,thin] (0,6.7)--(5.5,6.7);
\draw[blue] (2.75,6.9) node{$m$};

\draw[<->,blue,thin] (5.7,0)--(5.7,1);
\draw[blue] (6.2,0.5) node{$k_{s+\ell}$};

\end{tikzpicture}
\caption{The tiling of $[0,m]\times [0, m + k_{s+\ell}]$ by paths with alternating groups of $n$ horizontal gaps and $k_\ell$ vertical gaps (here $m = 11$, $k_{s+\ell} = 2$ and $n = 6$). In light blue domains points on the same horizontal line belong to the same path, in yellow domains points on the same vertical line belong to the same path. }\label{fig4}
\end{figure}
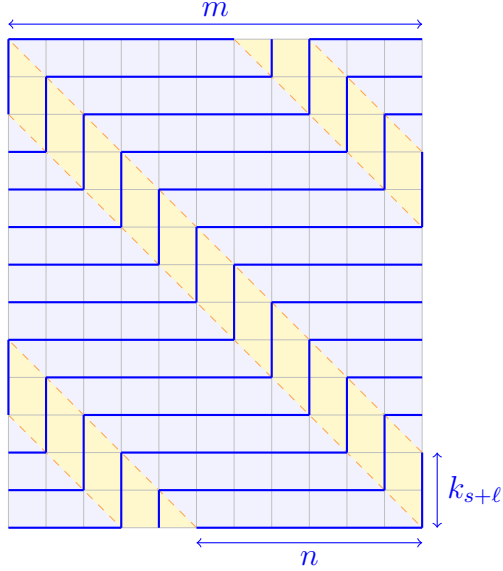

Recall that $m+1$ is the size of some sequence from the tiling for $\ell-1$, and $n = |T_{s+\ell-1}|$. By the induction hypothesis, $m + 1 \leq n+\sum_{i=1}^{\ell-1}k_{s+i} + 2 = \sum_{i = 1}^{s} k_i + 2\sum_{i=1}^{\ell-1}k_{s+i} + 2  = 2n - \sum_{i = 1}^{s} k_i + 2 < 2n + 1$ (here we use the inequality $s\ge 2$). Therefore, $m<2n$. Consequently, no path contains two entire horizontal segments of length $n$, and thus no path contains more than two vertical pieces. Therefore, the length of each path does not exceed $m + 2k_{s+\ell} \leq |T_{s+\ell-1}| +  \sum_{i=1}^{\ell-1}k_{s+i} + 1 + 2 k_{s + \ell} = |T_{s+\ell}| + \sum_{i=1}^{\ell}k_{s+i} + 1$.

Consider the bijection $\varphi: [0,m] \times [0, m+k_{s+\ell}] \to \{v_0,...,v_m\} \times [0, m+k_{s+\ell}]$, $\varphi: (x,y) \mapsto (v_x, y)$. It maps the tiling of $[0,m] \times [0, m+k_{s+\ell}]$ constructed above into the desired tiling of $\{v_0,...,v_m\} \times [0,m + k_{s+\ell}]$. Indeed, among any $|T_{s+\ell-1}| + k_{s+\ell}$ consecutive gaps of a path in the tiling of $[0,m] \times [0, m+k_{s+\ell}]$ there are exactly $|T_{s+\ell-1}|$ horizontal gaps. Consequently, orthogonal projection of the segment in the image of such a path containing consecutive $|T_{s+\ell-1}| + k_{s+\ell} + 1$ points on the $x$-axis is equal to $|T_{s+\ell-1}|+1$ consecutive points of the sequence $(v_0,...,v_m)$, i.e., by the induction hypothesis, a tile with gap set $T_{s+\ell-1}.$ Besides, such a segment contains exactly $k_{s+\ell}$ gaps equal to $e_2,$ same as its preimage.

If  $(v_0, ..., v_n)$  is a $T_{s+\ell-1}$-homogeneous sequence of cardinality exactly $|T_{s+\ell-1}|+1,$ then it is just a tile with gap set $T_{s+\ell-1}.$ 
In the proof of Lemma~\ref{lem1} in the induction step we have constructed a tiling of $\{v_0, ..., v_n\} \times [0, f_{n,k_{s+\ell}}(n+1) - 1]$ by paths of type $V_{s+\ell}$.

Thus, for any $T_{s+\ell-1}$-homogeneous sequence $\{v_0,...,v_m\}$ from the tiling, given by the induction hypothesis, there exists some positive integer $y(m)$ ($y = m+k_{\ell+s}+1$ if $m>n$ and $y = f_{n,k_{s+\ell}}(n+1)$ if $m = n$) such that there is a tiling of $\{v_0,...,v_m\} \times [0,y-1]$ by $V_{s+\ell}$-homogeneous paths. From these tilings we construct (similarly to the induction step in the proof of Lemma~\ref{lem1}) a tiling by such paths of the rectangle $[0,L]\times [0, h-1]$, where $h = \text{lcm}(f_{n,k_{s+\ell}}(n+1), m_1+k_{\ell+s}+1, ..., m_{q}+k_{\ell+s}+1)$ and $m_1+1,...,m_q+1$ are distinct cardinalities of sequences that tile $[0,L]$ such that $m_i > n$. In addition, the point $(L, h-1)$ corresponds to the point $(v_m, m+k_{\ell+s})$ from the tiling of $\{v_0,...,v_m\} \times [0,m+k_{s+\ell}]$ with $m>n$. Thus, the path containing $(L, h-1)$ consists of strictly more than $|T_{s+\ell}|+1$ points.

Note that, by the induction hypothesis, the homogeneous sequence containing the point $L$ is strictly longer than one tile. So we may remove this point, and the sequence remains $T_{s+\ell-1}$-homogeneous. Thereby, we obtain a tiling of $[0,L-1]$ by $T_{s+\ell-1}$-homogeneous sequences. Then, we may apply the above construction to this tiling of $[0,L-1]$ instead of the initial tiling of $[0,L]$. This gives a tiling of $[0,L-1]\times [0, h'-1]$  by $V_{s+\ell}$-homogeneous paths. Here $h' = \text{lcm}(f_{n,k_{s+\ell}}(n+1), m_1'+k_{\ell+s}+1, ..., m_{q'}'+k_{\ell+s}+1),$ where $m_1'+1,...,m_{q'} '+1$ are distinct cardinalities of sequences that tile $[0,L-1]$ with $m_i' > n$. Also, the lengths of the paths do not exceed $|T_{s+\ell}| + \sum_{i=1}^{\ell}k_{s+i} + 1$ in this tiling as well. 

Set $g_{s+\ell} = L^2 + 1.$ Then, provided $d_{s+\ell} \geq g_{s+\ell}$, by Proposition~\ref{lem0}, the rectangle $[0, d_{s+\ell} - 1]\times [0, \text{lcm}(h,h')-1]$ can be partitioned into translates of $[0,L-1]\times [0, h'-1]$ and $[0,L]\times [0, h-1]$ so that its top right corner lies in a copy of $[0,L]\times [0, h-1]$. 

The bijection $(a,b)\to a+bd_{s+\ell}$ between this rectangle and an interval 
maps the tiling of $[0, d_{s+\ell} - 1]\times [0, \text{lcm}(h,h')-1]$ into a tiling of an interval. $V_{s + \ell}$-homogeneous paths are mapped into  $T_{s+\ell}$-homogeneous sequences that have all the desired properties.
\end{proof}

Now we are ready to prove Theorem~\ref{thm1}. In Lemma~\ref{lem2} it is proved that there exists a tiling of $[0,L]$ by $T_{s+p-1}$-homogeneous sequences of cardinalities that do not exceed $|T_{s+p-1}| + \sum_{i=1}^{p-1}k_{s+i} + 2$.

Let $(v_0,...,v_m)$ be one of these sequences. Then $|T_{s+p-1}| + 1 \leq m+1 \leq |T_{s+p-1}| + \sum_{i=1}^{p-1}k_{s+i} + 2 \leq |T_{s+p-1}|+k_{s+p}+1.$ The last inequality follows from the assumption in the theorem. By Proposition~\ref{propnaka}, there exists a tiling of $[0,m]\times [0,f_{|T_{s+p-1}|,k_{s+p}}(m+1)-1]$ by paths of type $\{e_1^{(|T_{s+p-1}|)}, e_2^{(k_{s+p})}\}.$
The bijection $\varphi: (x,y) \mapsto (v_x, y)$ maps this tiling into a tiling of $\{v_0,...,v_m\} \times [0,f_{|T_{s+p-1}|,k_{s+p}}(m+1)-1]$ by paths of type $V_{s+p}$ (cf. Fig.~\ref{fig5}). This holds because $\varphi$ preserves vertical gaps, and the projection of the image of each path on the $x$-axis consists of $|T_{s+p-1}|+1$ consecutive points of $(v_0,...,v_m)$, i.e. a tile with gap set $T_{s+p-1}.$ Arguing as before, this implies that there exists a tiling  of $[0,L] \times [0,\text{lcm}(f_{|T_{s+p-1}|,k_{s+p}}(m_1+1),...,f_{|T_{s+p-1}|,k_{s+p}}(m_u+1))-1]$ by paths of type $V_{s+p}$. Here $(m_i+1)$'s are distinct cardinalities of the sequences from the tiling of $[0,L].$ 

\begin{figure}[h]
\begin{tikzpicture}
    \filldraw[blue!15] (0,0) rectangle (1,2);
    \filldraw[red!15] (1,0) rectangle (3,2);
    \filldraw[green!15] (3,0) rectangle (7,2);
    \filldraw[blue!15] (7,0) rectangle (8,2);
    \filldraw[red!15] (8,0) rectangle (10,2);
    \foreach \x in {0,1,3,7,8,10} \filldraw (\x,0) circle[radius=2pt];
    \foreach \x in {0,1,3,7,8,10} \filldraw (\x,1) circle[radius=2pt];
    \foreach \x in {0,1,3,7,8,10} \filldraw (\x,2) circle[radius=2pt];
    \foreach \x in {2,4,5,6,9} \filldraw[gray!80] (\x,0) circle[radius=2pt];
    \foreach \x in {2,4,5,6,9} \filldraw[gray!80] (\x,1) circle[radius=2pt];
    \foreach \x in {2,4,5,6,9} \filldraw[gray!80] (\x,2) circle[radius=2pt];
    \draw (3,0)--(10,0)--(10,2);
    \draw (1,0)--(1,1)--(8,1)--(8,2);
    \draw (0,0)--(0,2)--(7,2);
\end{tikzpicture}
\caption{The tiling of $\{v_0,...,v_m\}\times [0,2]$ by paths of type $\{e_1^{(1)}, (2e_1)^{(1)}, (4e_1)^{(1)}, e_2^{(2)}\}$, where $(v_0,...,v_m)$ in a $\{1^{(1)}, 2^{(1)}, 4^{(1)}\}$ - homogeneous sequence}\label{fig5}
\end{figure}
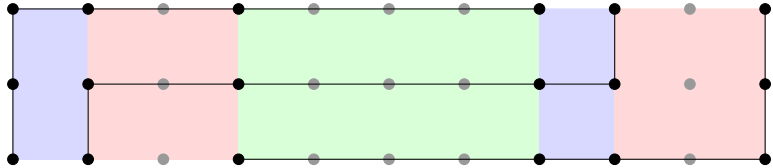

As before, we can remove the point $L$ from the tiling of $[0,L]$. Using this new tiling, we may obtain a tiling 
by paths of type $V_{s+p}$ of $[0,L-1] \times [0,\text{lcm}(f_{|T_{s+p-1}|,k_{s+p}}(m_1'+1),...,f_{|T_{s+p-1}|,k_{s+p}}(m_{u'} '+1))-1]$, where $(m_i'+1)$'s are distinct cardinalities of the sequences from the tiling of $[0,L-1].$ 

Provided $d_{s+p} \geq L(L-1)$, we can (by Proposition~\ref{lem0}) construct a tiling of a rectangle of width $d_{s+p}$ using the two tilings constructed above. The map $(a,b)\to a+d_{s+p}b$ maps this tiling into the desired tiling of the interval.
Theorem~\ref{thm1} is proved.

\section{Acknowledgements}
This work was supported by a grant for research centers in the field of artificial intelligence, provided by the Analytical Center for the Government of the Russian Federation in accordance with the subsidy agreement (agreement identifier 000000D730321P5Q0002) and the agreement with the Moscow Institute of Physics and Technology dated November 1, 2021 No. 70-2021-00138.

\end{document}